\documentclass{birkjour}
\usepackage{mathrsfs}
\usepackage{amsfonts}
 \newtheorem{thm}{Theorem}[section]
 \newtheorem{cor}[thm]{Corollary}
 \newtheorem{lem}[thm]{Lemma}
 
 \theoremstyle{definition}
 \newtheorem{defn}[thm]{Definition}
 \theoremstyle{remark}
 \newtheorem{rem}[thm]{Remark}
 
 \numberwithin{equation}{section}
 
 \newcommand{\Z}{\mathbb{Z}}
\newcommand{\Hom}{\mathrm{Hom}}
\newcommand{\C}{\mathbb{C}}

\newcommand{\Tr}{\mathrm{Tr} }
\newcommand{\undcal}[2]{#1_{\mathcal{#2}}}

\newcommand{\pos}[2]{\mathscr{#1} \! \left(\mathcal{#2}\right)}

\newcommand{\derham}[2]{H^{#1} \! \left(#2\right)}
\newcommand{\Htot}[1]{H^* \! \left(#1\right)}
\newcommand{\cd}{\mathrm{cd}}
\newcommand{\eul}{E \!}
\newcommand{\lef}{L}

\begin{document}

%
%
%
%
%
%
%
%
%

\title[Cohomology of the toric arrangement associated with $A_n$]
 {Cohomology of the toric arrangement \\associated with $A_n$}

\author[O. Bergvall]{Olof Bergvall}

\address{%
Department of Electrical Engineering, Mathematics and Science\\
University of G\"avle\\
801 76 G\"avle\\
Sweden\\
Tel.: +426-64 89 65}

\email{olof.bergvall@hig.se}

\subjclass{Primary 20F55; Secondary 52C35, 54H25}

\keywords{Arrangements, Cohomology, Fixed points, Root systems, Weyl groups}


\begin{abstract}
We compute the total cohomology of the complement of the toric arrangement associated to
    the root system $A_n$ as a representation of the corresponding Weyl group via fixed point theory
    of a ``twisted'' action of the group. We also provide several proofs of
    an explicit formula for the Poincar\'e polynomial
    of the complement of the toric arrangement associated to $A_n$.
\end{abstract}

\maketitle

\section{Introduction}
An arrangement $\mathcal{A}$ is a finite set of closed subvarieties of
a variety $X$.
A finite group $\Gamma$ of automorphisms of $X$, which fixes $\mathcal{A}$
as a set, acts on the complement $\undcal{X}{A}$ of $\mathcal{A}$.
The group $\Gamma$ therefore
also acts on the de Rham cohomology groups $\derham{i}{\undcal{X}{A}}$
which in this way become $\Gamma$-representations.
It is an interesting, but often hard, problem to determine these representations. 
A somewhat easier, but still interesting, problem is to
determine the total cohomology $H^*(\undcal{X}{A})$ as a representation
of $\Gamma$.

In Section~\ref{arrsec} we consider this problem and, generalizing ideas
of Felder and Veselov \cite{felderveselov}, we develop a method for computing $H^*(\undcal{X}{A})$ 
via fixed point theory,
provided that the cohomology groups $\derham{i}{\undcal{X}{A}}$ have
sufficiently nice mixed Hodge structure. It is known that many important
classes of arrangements are of this type and in Section~\ref{rootsec}
we apply our method in the case of a toric arrangement associated to the root
system $A_n$ in order to compute the total cohomology as a representation
of the Weyl group of $A_n$. 

\pagebreak[4]
\begin{thm}
 \label{antotcohthm}
  Let $W_{A_n}$ be the Weyl group of the root system $A_n$.
  Then the total cohomology of the complement $X_{A_n}$
  of the toric arrangement associated to $A_n$ is the $W_{A_n}$-representation
  \begin{equation*}
   H^*(X_{A_n}) = \mathrm{Reg}_{W_{A_n}} + n \cdot \mathrm{Ind}_{\langle s \rangle}^{W_{A_n}}(\mathrm{Triv}_{\langle s \rangle}),
  \end{equation*}
 where $\mathrm{Reg}_{W_{A_n}}$ is the regular representation of $W_{A_n}$
 and $\mathrm{Ind}_{\langle s \rangle}^{W_{A_n}}(\mathrm{Triv_{\langle s \rangle}})$ denotes
 the representation of $W_{A_n}$ induced up from the trivial representation
 of the subgroup generated by the simple reflection $s=(12)$.
 \end{thm}

In Section~\ref{poincsec} we continue our study of the complement
of the toric arrangement associated to $A_n$ but we now forget about
the action of the Weyl group and focus on the Poincar\'e polynomial, i.e.
we study the dimensions of the individual cohomology groups.

\begin{thm}
 \label{anpoincpolprop}
 The Poincar\'e polynomial of the complement $X_{A_n}$
  of the toric arrangement associated to $A_n$ is given by
  \begin{equation*}
  P(X_{A_n},t) = \prod_{i=1}^n (1+(i+1) \cdot t).
  \end{equation*}
 \end{thm}

\subsection{Previous work}
Motivated by questions related to the braid group, 
the study of the cohomology of arrangement complements was initiated by Arnol'd in
\cite{arnold} where he computed the cohomology ring of the complement of the hyperplane arrangement
associated to $A_n$. In particular, he found the formula
\begin{equation}
 \label{arnoldformula}
  \prod_{i=1}^n (1+i \cdot t).
 \end{equation}
for the Poincar\'e polynomial. This formula was later generalized to general root systems
by Orlik and Solomon \cite{orliksolomon}. Theorem~\ref{anpoincpolprop} is thus a toric analogue
of Arnol'd's formula. It is certainly well-known to experts (it is in fact implicit already
in the work of Arnol'd) but the explicit statement seems to be missing from the literature.
The purpose of Section~\ref{poincsec} is thus to record the result as well as to discuss
its proof from various points of view.

The study of Arnol'd was continued in a slightly different direction by
Brieskorn \cite{brieskorn} who described
the action of the Weyl group of $A_n$ on the cohomology ring of the complement
of the associated hyperplane arrangement. These results were later improved and extended
by Lehrer \cite{lehrer} and Lehrer and Solomon \cite{lehrersolomon}.
The study of the cohomology of complements of toric arrangements was initiated by Looijenga in
\cite{looijenga} as part of his computation of the cohomology of the moduli space $\mathcal{M}_{3}$
of smooth curves of genus three (it should be mentioned that this paper contains some mistakes which
later have been corrected by Getzler and Looijenga himself in \cite{getzlerlooijenga} and by De Concini and
Procesi in \cite{deconciniprocesi2}). 
Theorem~\ref{antotcohthm} is a toric analogue of the works of Brieskorn \cite{brieskorn},
Lehrer \cite{lehrer} and Lehrer-Solomon \cite{lehrersolomon}.
It should be pointed out that
it can be derived from results of Getzler \cite{getzler} (who was investigating moduli spaces of
rational curves), as well as from results of Gaiffi \cite{gaiffi}, 
Mathieu \cite{mathieu} and Robinson and Whitehouse \cite{robinsonwhitehouse} (who all were more directly interested in arrangements).
Relevant formulas (from the point of view of representation stability) can also be
found in the work of Hersh and Reiner \cite{hershreiner}. 
Our main contribution is thus the new, fixed point theoretic proof.
It also seems to us that our method is the most promising for generalizations, e.g.
 to make computations for other root systems of classical type.
 
\section{Arrangements}
\label{arrsec}
Unless otherwise specified, we shall always work over the complex numbers.

\begin{defn}
\label{arrdefn}
 Let $X$ be a variety. An \emph{arrangement} $\mathcal{A}$ in $X$ is  a finite set 
 $\{A_i\}_{i \in I}$ of closed subvarieties of $X$, where $I$ is a finite index set.
\end{defn}

Given an arrangement $\mathcal{A}$ in a variety $X$ one may define its \emph{cycle}
\begin{equation*}
 \undcal{D}{A} = \bigcup_{i \in I} A_i \subset X,
\end{equation*}
and its \emph{open complement}
\begin{equation*}
 \undcal{X}{A} = X \setminus \undcal{D}{A}.
\end{equation*}
The variety $\undcal{X}{A}$, or rather its de Rham cohomology groups $\derham{i}{\undcal{X}{A}}$,
will be our main object of study. We shall always consider cohomology with rational coefficients.

Let $\Gamma$ be a finite group of automorphisms of $X$ that stabilizes
$\mathcal{A}$ as a set. The action of $\Gamma$ induces actions on
$\undcal{X}{A}$ and on the cohomology groups $\derham{i}{\undcal{X}{A}}$. Each individual cohomology
group $\derham{i}{\undcal{X}{A}}$ thus becomes a $\Gamma$-representation 
and, therefore, so does the total cohomology $\Htot{\undcal{X}{A}}$.
We shall now explain a method to determine $\Htot{\undcal{X}{A}}$ in a large class of interesting cases,
including arrangements of hyperplanes and toric arrangements, generalizing
ideas of Felder and Veselov \cite{felderveselov}.

\subsection{The total cohomology}
Let $\mathcal{A}$
be an arrangement in a variety $X$ and let $\Gamma$ be a finite
group of automorphisms of $X$ that fixes $\mathcal{A}$ as a set.
The group $\Gamma$ will then act on the individual cohomology groups of $\undcal{X}{A}$
and thus on the total cohomology
\begin{equation*}
 \Htot{\undcal{X}{A}} := \bigoplus_{i \geq 0} \derham{i}{\undcal{X}{A}}.
\end{equation*}
The value of the \emph{total character} at $g \in \Gamma$ is defined as
\begin{equation*}
 P(\undcal{X}{A})(g) : = \sum_{i \geq 0} \Tr \left(g, \derham{i}{\undcal{X}{A}} \right),
\end{equation*}
and the \emph{Lefschetz number} of $g \in \Gamma$ is defined as
\begin{equation*}
 \lef (\undcal{X}{A})(g) := \sum_{i \geq 0} (-1)^i \cdot \Tr \left(g, \derham{i}{\undcal{X}{A}} \right).
\end{equation*}
Let $\undcal{X}{A}^g$ denote the fixed point locus of $g \in \Gamma$.
Lefschetz fixed point theorem, see \cite{brown}, then states that
the Euler characteristic $\eul \left(\undcal{X}{A}^g\right)$ of $\undcal{X}{A}^g$
equals the Lefschetz number of $g$, i.e.
\begin{equation*}
 \eul \left( \undcal{X}{A}^g \right) = \lef (\undcal{X}{A})(g).
\end{equation*}

We now specialize to the case when each cohomology group $\derham{i}{\undcal{X}{A}}$ 
has mixed Hodge structure of the form
\begin{equation*}
 \derham{i}{\undcal{X}{A}} = \bigoplus_{j \equiv i \, \mathrm{mod} \, 2} W_{j,j}\derham{i}{\undcal{X}{A}}
\end{equation*}
where $W_{j,j}\derham{i}{\undcal{X}{A}}$ denotes the part of $\derham{i}{\undcal{X}{A}}$ of Tate type $(j,j)$
and where each element of $\mathcal{A}$ is fixed by complex conjugation.
It is known, through work of Brieskorn \cite{brieskorn} and Looijenga \cite{looijenga}, 
that if $\mathcal{A}$ is a hyperplane arrangement or a toric arrangement, then $\derham{i}{\undcal{X}{A}}$ 
has Tate type $(i,i)$ so many interesting examples are of this type (see \cite{dimcalehrer} for a more complete discussion).
We define an action of $\Gamma \times \mathbb{Z}_2$ on $X$
by letting $(g,0) \in \Gamma \times \mathbb{Z}_2$ act as $g \in \Gamma$ and $(0,1) \in\Gamma \times \mathbb{Z}_2$
act by complex conjugation. 
Since $\mathcal{A}$ is fixed under conjugation, this gives an
action on $\undcal{X}{A}$. We write $\bar{g}$ to denote the element
$(g,1) \in \Gamma \times \mathbb{Z}_2$.

Since $\derham{i}{\undcal{X}{A}}$ only has parts with Tate type congruent to $(i,i)$ modulo $2$, complex conjugation
acts as $(-1)^i$ on $\derham{i}{\undcal{X}{A}}$. We thus have
\begin{align*}
 \lef \left(\undcal{X}{A}\right)(\bar{g}) & = \sum_{i \geq 0} (-1)^i \Tr \left(\bar{g}, \derham{i}{\undcal{X}{A}} \right) = \\
 & = \sum_{i \geq 0} (-1)^i \cdot (-1)^i \cdot  \Tr \left(g, \derham{i}{\undcal{X}{A}} \right) = \\
 & = P\left(\undcal{X}{A}\right)(g).
\end{align*}
Since $\lef \left(\undcal{X}{A}\right)(\bar{g}) = \eul \left(\undcal{X}{A}^{\bar{g}}\right)$
we have proved the following lemma.

\begin{lem}
\label{FVlemma}
 Let $X$ be a smooth variety and let $\mathcal{A}$ be
 an arrangement in $X$ which is fixed by complex conjugation and such that
 $\derham{i}{\undcal{X}{A}}$ only has parts of Tate type congruent to $(i,i)$ modulo $2$.
 Let $\Gamma$ be a finite group which acts on $X$ as automorphisms
 and which fixes $\mathcal{A}$ as a set. Then
 \begin{equation*}
 P\left(\undcal{X}{A}\right)(g)= \eul \left(\undcal{X}{A}^{\bar{g}}\right).
 \end{equation*}
\end{lem}

\section{Toric arrangements associated to root systems}
\label{torarrsec}
Classically, arrangements of hyperplanes have been given most attention.
However, in the past two decades an increasing number of authors have considered 
also toric arrangements and they have been studied from the point of view of
geometry, topology, algebra and combinatorics.

\begin{defn}
 Let $X$ be an $n$-torus. An arrangement $\mathcal{A}$ in $X$
 is called a \emph{toric arrangement} if each element of $\mathcal{A}$
 is a subtorus.
\end{defn}

\begin{rem}
We shall only be interested in the case where each subtorus in the arrangement has codimension
one, i.e. where the arrangement is divisorial, and write ``toric arrangement'' to
mean ``divisorial toric arrangement''.
\end{rem}

Let $\Phi$ be a root system, let $\Delta= \{\beta_1, \ldots, \beta_n\}$ be a set of simple
roots and let $\Phi^+$ be the set of positive roots of $\Phi$ with respect
to $\Delta$. We think of $\Phi$ as a set of vectors in some 
real Euclidean vector space $V$ and we let $M$ be the $\Z$-linear span of $\Phi$. 
Thus, $M$ is a free $\Z$-module of finite rank $n$.

Define $X=\Hom(M,\C^*) \cong (\C^*)^n$. The Weyl group $W_{\Phi}$ of $\Phi$ acts
on $X$ from the right by precomposition, i.e.
\begin{equation*}
 (\chi.g)(v) = \chi(g.v).
\end{equation*}
For each $\alpha \in \Phi$ we define
\begin{equation*}
 A_{\alpha} = \{\chi \in X | \chi(\alpha) = 1\}.
\end{equation*}
We thus obtain an arrangement of hypertori in $X$
\begin{equation*}
 \mathcal{A}_{\Phi}=\left\{ A_{\alpha} \right\}_{\alpha \in \Phi}.
\end{equation*}
To avoid cluttered notation we
shall write $X_{\Phi}$ instead of the more cumbersome $X_{\mathcal{A}_{\Phi}}$.

Let $\chi \in X$. We introduce the notation $\chi(\beta_i)=z_i$ for the simple roots
$\beta_i$, $i=1, \ldots, n$. The coordinate ring of $X$ is then
\begin{equation*}
 \mathbb{C}\left[X\right] = \mathbb{C}[z_1, \ldots, z_n,z_1^{-1}, \ldots, z_n^{-1}].
\end{equation*}
If $\alpha$ is a root, there are integers $m_1, \ldots, m_n$ such that
\begin{equation*}
 \alpha = m_1 \cdot \beta_1 +  \cdots + m_n \cdot \beta_n.
\end{equation*}
With this notation we have that $\chi(\alpha)=1$ if and only if
\begin{equation*}
 z_1^{m_1} z_2^{m_2} \cdots z_n^{m_n} = 1.
\end{equation*}

\section{Root systems of type $A$}
\label{rootsec}
The root system $A_n$ is most naturally viewed in an $n$-dimensional
subspace of $\mathbb{R}^{n+1}$. Denote the $i$th coordinate vector of $\mathbb{R}^{n+1}$ by $e_i$.
The roots $\Phi$ can then be chosen to be
\begin{equation*}
 \alpha_{i,j} = e_i - e_j, \quad i \neq j.
\end{equation*}
A choice of positive roots is
\begin{equation*}
 \alpha_{i,j} = e_i - e_j, \quad i < j,
\end{equation*}
and the simple roots with respect to this choice of positive roots
are
\begin{equation*}
 \beta_{i} = e_i - e_{i+1}, \quad i=1, \ldots, n.
\end{equation*}
The Weyl group of group $A_n$ is isomorphic to the symmetric group $S_{n+1}$
and an element of $S_{n+1}$ acts on an element in $M=\mathbb{Z}\left\langle \Phi \right\rangle$ by permuting
the indices of the coordinate vectors in $\mathbb{R}^{n+1}$.

\subsection{The total character}
In this section we shall compute the value of the total character at any element
$g \in W_{A_n}$. This will determine the total cohomology $\derham{*}{X_{A_n}}$
as an $W_{A_n}$-representation.
Although we have not pursued this, similar methods should allow the
computation of $\derham{*}{X_{\Phi}}$ also in the case of
root systems of type $B_n$, $C_n$ and $D_n$.

\begin{lem}
\label{gthantwolem}
 Let $W_{A_n}$ be the Weyl group of $A_n$ and suppose that $g \in W_{A_n}$ 
 has a cycle of length greater than two. 
 Then $X_{A_n}^{\bar{g}}$ is empty.
\end{lem}

\begin{proof}
 The statement only depends on the conjugacy class of $g$ so suppose
 that $g$ contains the cycle $(1,2,\ldots,s)$, where $s \geq 3$. 
 We then have
 \begin{align*}
  & g.\beta_1=e_2-e_3 =\beta_2,\\
  & g.\beta_2=e_3-e_4 =\beta_3,\\
  & \vdots \\
  & g.\beta_{s-2}=e_{s-1}-e_s = \beta_{s-1},\\
  & g.\beta_{s-1} = e_s-e_1 = -(\beta_1+\ldots+\beta_{s-1}).
 \end{align*}
 If $\overline{g}.\chi=\chi$ we must have
 \begin{align*}
  & z_i = \overline{z}_{i+1} \quad \text{for} \quad i= 1, \ldots, s-2, \tag{1}\\
  & z_{s-1}=\overline{z}_1^{-1}\overline{z}_2^{-1} \cdots \overline{z}_{s-1}^{-1}. \tag{2}
 \end{align*}
 We insert (1) into (2) and take absolute values to obtain
 $|z_1|^s=1$. We thus see that $|z_1|=1$. Since we have $z_2 = \overline{z}_1$ it follows that
 \begin{equation*}
 \chi(\alpha_{1,3})=\chi(\beta_1+\beta_2)=z_1 \cdot z_2 =z_1 \cdot \overline{z}_1 = |z_1|^2=1.
 \end{equation*}
 Thus, $\chi$ lies in $X_{\alpha_{13}}$ so $X_{\Phi}^{\overline{g}}$ is empty.
\end{proof}

%

 If we apply Lemma~\ref{FVlemma} to Lemma~\ref{gthantwolem} we obtain the following corollary.

\begin{cor}
 If $g$ is an element of the Weyl group of $A_n$ such that $g^2 \neq \mathrm{id}$, then
 \begin{equation*}
  P(X_{A_n})(g)=0.
 \end{equation*}
\end{cor}

We thus know the total character of all elements in the Weyl group of $A_n$ of order greater
than $2$. We shall therefore turn our attention to the involutions.

\begin{lem}
 \label{manytranslem}
 If $g$ is an element of the Weyl group of $A_n$ of order $2$ which is not a reflection, then 
 \begin{equation*}
  E(X_{A_n}^{\bar{g}})= P(X_{A_n})(g)=0.
 \end{equation*}
\end{lem}
 
 \begin{proof}
 Let $k>1$ and consider the element $g=(1,2)(3,4) \cdots (2k-1,2k)$. We define a new basis
 for $M$:
 \begin{align*}
  & \gamma_i = \beta_i + \beta_{i+1}, &   i=1, \ldots, 2k-2,  & \,\\
  & \gamma_j = \beta_j & j=2k-1, \ldots,n & \,
 \end{align*}
 Then
 \begin{equation*}
  g.\gamma_{2i-1}=\gamma_{2i}, \quad \text{and} \quad g.\gamma_{2i}=\gamma_{2i-1}, 
 \end{equation*}
 for $i=1, \ldots k-1$, 
 \begin{equation*}
  g.\gamma_{2k-1}=-\gamma_{2k-1}, \quad g.\gamma_{2k}=\gamma_{2k-1}+\gamma_{2k},
 \end{equation*}
 and $g.\gamma_i=\gamma_i$ for $i>2k$. If we put $\chi(\gamma_i)=t_i$, then $X_{A_n}^{\bar{g}} \subseteq X_{A_n}$ is given by the equations
 \begin{equation*}
 \begin{array}{lcll}
  t_{2i-1} & = & \bar{t}_{2i} &  i=1, \ldots, k-1, \\
  t_{2k-1} & = & \bar{t}_{2k-1}^{-1}, & \, \\
  t_{2k}   & = & \bar{t}_{2k-1} \cdot \bar{t}_{2k}, & \, \\
  t_{i}    & = & \bar{t}_i, & i = 2k+1, \ldots, n.
  \end{array}
 \end{equation*}
 Thus, the points of $X_{A_n}^{\bar{g}}$ have the form 
 \begin{equation*}
 (t_1, \bar{t}_1, t_2, \bar{t}_2, \ldots, t_{k-1}, \bar{t}_{k-1},s,s^{-1/2} \cdot r, t_{2k}, \ldots, t_{n}),
 \end{equation*}
 where $s \in S^1 \setminus \{1\} \subset \mathbb{C}$, $r \in \mathbb{R}^+$ and $t_i \in \mathbb{R}$ for $i=2k, \ldots, n$.
 
 We can now see that each connected component of $X_{A_n}^{\bar{g}}$ is homeomorphic
 to $\mathcal{C}_{k-1}(\mathbb{C}\setminus \{0,1 \}) \times (0,1)  \times \mathbb{R}^{n-2k+1}$,
 where $\mathcal{C}_{k-1}(\mathbb{C}\setminus \{0,1 \})$ is the configuration space
 of $k-1$ points in the twice punctured complex plane. This space is in turn homotopic
 to the configuration space $\mathcal{C}_{k+1}(\mathbb{C})$ of $k+1$ points in the complex
 plane. The space $\mathcal{C}_{k+1}(\mathbb{C})$ is known to have Euler characteristic zero
 for $k \geq 1$.
 \end{proof}

 \begin{rem}
  Note that it is essential that we use ordinary cohomology in the above
  proof, since compactly supported cohomology is not homotopy invariant.
  However, since $X_{\Phi}$ satisfies Poincar\'e duality, the corresponding
  result follows also in the compactly supported case.
 \end{rem}
 
 We now turn to the reflections.
 
 \begin{lem}
 \label{reflectionlem}
  If $g$ is a reflection in the Weyl group of $A_n$, then 
  \begin{equation*}
   E(X_{A_n}^{\bar{g}})=P(X_{A_n})(g)=n!.
  \end{equation*}
 \end{lem}
 
 \begin{proof}
  Let $g=(1,2)$. We then have
 \begin{align*}
  g.\beta_1 & = -\beta_1, \\
  g.\beta_2 & =  \beta_1+\beta_2,\\
  g.\beta_{i} & =  \beta_i,  & i>2.
 \end{align*}

 This gives the equations
 \begin{align*}
   z_1 & = \overline{z}_1^{-1}, \\
   z_2 & = \overline{z}_1 \cdot \overline{z}_{2}, \\
   z_i & = \overline{z}_i, & i>2.
 \end{align*}
 Thus $z_1 \in S^1 \setminus \{1\}$, $z_2$ is not real and satisfies $z_2 =\overline{z}_1 \cdot \overline{z}_{2}$
 so we choose $z_2$ from a space isomorphic to $\mathbb{R}^*$.
 Hence, $X_{A_n}^{\bar{g}} \cong [0,1] \times \mathbb{R}^* \times Y$
 where $Y$ is the space where the last $n-2$ coordinates  $z_3, \ldots,z_n$ takes their values.
 
 These coordinates satisfy $z_i = \overline{z}_i$, i.e. they are real.
 We begin by choosing $z_3$. Since $\chi(e_i - e_j) \neq 0, 1$ we
 need $z_3 \neq 0, 1$. We thus choose $z_3$ from $\mathbb{R}\setminus \{0,1\}$.
 We then choose $z_4$ in $\mathbb{R}\setminus \{0,1,\frac{1}{z_3}\}$,
 $z_5$ in $\mathbb{R}\setminus \{0,1,\frac{1}{z_4}, \frac{1}{z_3 \cdot z_4}\}$ and so on. 
 In the $i$th step we have $i$ components to choose from. Thus, $Y$ consists of
 \begin{equation*}
 3 \cdot 4 \cdots n = \frac{n!}{2}
 \end{equation*}
 components, each isomorphic to $\mathbb{R}^{n-2}$. Hence, $E(Y)=\frac{n!}{2}$
 and it follows that
 \begin{equation*}
  E(T_{\Phi}^{\bar{g}})= E([0,1]) \cdot E(\mathbb{R}^*) \cdot E(Y)= n!.
 \end{equation*}
 \end{proof}

 It remains to compute the value of the total character at the identity element.
 
 \begin{lem}
 \label{idlem}
  $E(X_{A_n}^{\bar{\mathrm{id}}})=P(X_{A_n})(\mathrm{id})=\frac{(n+2)!}{2}$
 \end{lem}
 
 \begin{proof}
  The proof is a calculation similar to that in
  the proof of Lemma~\ref{reflectionlem}. We note that the equations
  for $X_{A_n}^{\bar{\textrm{id}}}$ are $z_i = \bar{z}_i$, $i=1, \ldots, n$
  so the computation of $E(X_{A_n}^{\bar{\textrm{id}}})$ is essentially the same
  as that for $E(Y)$ above. The difference is that we have $n$ steps
  and in the $i$th step we have $i+1$ choices. This gives the result.
 \end{proof}
 
 Lemmas~\ref{gthantwolem}, \ref{manytranslem}, \ref{reflectionlem} and \ref{idlem}
 together determine the character of $W_{A_{n}}$ on $\derham{*}{X_{A_n}}$.
 The representation $\mathrm{Ind}_{\langle s \rangle}^{W_{A_n}}(\mathrm{Triv}_{\langle s \rangle})$
 takes value $\frac{(n+1)!}{2}$ on the identity element, $2(n-1)!$ on
 transpositions and is zero elsewhere. Since the character of $H^*(X_{A_n})$ takes the value
 $(n+2)!/2$ on the identity, $n!$ on transpositions and is zero elsewhere
 we see that Theorem~\ref{antotcohthm} holds.
 
 \begin{rem}
 The corresponding calculation for the affine hyperplane case was first computed by Lehrer in \cite{lehrer}
 and later by Felder and Veselov in \cite{felderveselov}. In the hyperplane case, the total
 cohomology turned out to be $2 \, \mathrm{Ind}_{\langle s \rangle}^{W_{A_n}}(\mathrm{Triv_{\langle s \rangle}})$,
 where $s$ is a transposition, i.e. the cohomology is twice the representation induced up from the trivial
 representation of the subgroup generated by a transposition. Thus, the representation 
 $\mathrm{Ind}_{\langle s \rangle}^{W_{A_n}}(\mathrm{Triv}_{\langle s \rangle})$ accounts for most of the
 cohomology also in the hyperplane case.
 \end{rem}
 
 \section{The Poincar\'e polynomial}
 \label{poincsec}
 In this section we shall see that the Poincar\'e polynomial of the complement of the toric arrangement associated
 to $A_n$ satisfies the formula given in 
 Theorem~\ref{anpoincpolprop}. We give several proofs of this result.
 Before giving the first proof we remark that setting $t=1$ in Theorem~\ref{anpoincpolprop}
 gives another proof of Lemma~\ref{idlem}.
 
 \begin{proof}[Proof 1]
  The key observation in the first proof is that $X_{A_n}$ is isomorphic to the moduli space $\mathcal{M}_{0,n+3}$
  of smooth rational curves marked with $n+3$ points. To see this, note that
  \begin{equation*}
   X_{A_n} = \{(x_0,x_1, \ldots, x_n) \in (\mathbb{C}^*)^{n+1}/\mathbb{C}^*| x_i \neq x_j\}.
  \end{equation*}
  We thus have that
  \begin{equation*}
   (x_0,x_1, \ldots, x_n) \mapsto [x_1/x_0,\ldots,x_n/x_0,0,1,\infty] \in (\mathbb{P}^1)^{n+3}
  \end{equation*}
  gives an isomorphism $X_{A_n} \to \mathcal{M}_{0,n+3}$.

  We count the number of points of $\mathcal{M}_{0,n+3}$ over a finite field $\mathbb{F}_q$ with $q$
  elements by choosing $n+3$ distinct points on $\mathbb{P}^1$ and dividing the result by the order
  of $\mathrm{PGL}(2,\mathbb{F}_q)$. We have $|\mathbb{P}^1(\mathbb{F}_q)|=q+1$ and $|\mathrm{PGL}(2,\mathbb{F}_q)| = (q+1)q(q-1)$.
  This gives
  \begin{equation*}
   |\mathcal{M}_{0,n+3}(\mathbb{F}_q)| = \frac{\prod_{i=1}^{n+3}(q+1-(i-1))}{(q+1)q(q-1)} = \prod_{i=1}^{n}(q-(i+1)).
  \end{equation*}
  Call the above polynomial $p(q)$.
 By results of Dimca and Lehrer \cite{dimcalehrer} we obtain the Poincar\'e polynomial
 of $\mathcal{M}_{0,n+3}$ as $(-t)^n \cdot p(-1/t)$.
 \end{proof}

 The Poincar\'e polynomial of $\mathcal{M}_{0,n}$ is well-known
 and the above proof is not new,
 although it seems to be missing in the literature. The advantage of the above
 proof is that it easily extends to make equivariant computations.
 Neither this is new and we refer the interested reader to Getzler \cite{getzler}
 for the results (although his methods are quite different).
 
 In the above proof, we computed something quite different from what we
 originally were interested in (the number of points of $\mathcal{M}_{0,n+3}$ over $\mathbb{F}_q$)
 and arrived at the desired result via a change of variables.
 This will also be the case in the second proof which has a more combinatorial flavour.
 
 \begin{proof}[Proof 2]
    Let $M_{A_n}(x,y)$ denote the arithmetic Tutte polynomial corresponding
  to the toric arrangement associated to $A_n$. By Theorem 5.11 of
  \cite{moci} we have 
  \begin{equation*}
  P(X_{A_n},t)=t^nM_{A_n}\left(\frac{2t+1}{t},0\right).
  \end{equation*}  
  We introduce the new variables $X=(x-1)(y-1)$ and $Y=y$ and define 
  $\psi_{A_n}(X,Y)=(-1)^nM_{A_n}(x,y)$.
  We then have
  \begin{equation*}  
  \psi_{A_n}\left(-\frac{t+1}{t},0 \right)=\left(- \frac{1}{t} \right)^n P(X_{A_n},t).
  \end{equation*}
  
  Let
  \begin{equation*}
   F(x,y) = \sum_{n \geq 0} \frac{x^n y^{\binom{n}{2}}}{n!}
  \end{equation*}
  and define the generating function
  \begin{equation*}
   \Psi_A(X,Y,Z) = 1 + X \cdot \sum_{n \geq 1} \psi_{A_{n-1}}(X,Y) \frac{Z^n}{n!}.
  \end{equation*}
  Then
  \begin{equation}
  \label{psiequation}
   \Psi_A \left( -\frac{t+1}{t},0,Z \right) = 1 + \left( - \frac{t+1}{t} \right) \cdot
   \sum_{n \geq 1} \left( - \frac{1}{t} \right)^{n-1} P(X_{A_{n-1}},t) \frac{Z^n}{n!}.
  \end{equation}
  
  By Theorem 1.14 of \cite{ardilacastillohenley} we have
  $\Psi_A(X,Y,Z)=F(Z,Y)^X$. We thus get
  \begin{align*}
   \Psi_A\left( -\frac{t+1}{t},0,Z \right) & = F(Z,0)^{-\frac{t+1}{t}} = \\
   \, & = (1+Z)^{-\frac{t+1}{t}} = \\
   \, & = \sum_{n \geq 0} \binom{-\frac{t+1}{t}}{n} Z^n = \\
   \, & = \sum_{n \geq 0} \frac{\left(-\frac{t+1}{t}\right)\left(-\frac{t+1}{t}-1\right) \cdots 
   \left(-\frac{t+1}{t}-n+1\right)}{n!}Z^n = \\
   \, & = 1 + \left(-\frac{t+1}{t}\right) \cdot
   \sum_{n \geq 1} \left(-\frac{1}{t}\right)^{n-1} \cdot \prod_{i=1}^{n-1}((i+1)t+1) \frac{Z^n}{n!}.
  \end{align*}
  We now get the result by comparing the above expression with Equation \ref{psiequation}.
 \end{proof}
 
 The advantage of the above proof is that it extends to other root systems of classical type.
 However, it does not seem to work well equivariantly.
 
 To give the third and final proof we need some new terminology.
 Given an arrangement $\mathcal{A}= \{A_i\}_{i \in I}$ in a variety $X$ we define
 the intersection poset of $\mathcal{A}$ as the set
 \begin{equation*}
  \pos{L}{A} = \{ \cap_{j \in J} A_j | J \subseteq I\}
 \end{equation*}
 ordered by reverse inclusion. 
 
 Let $\Phi$ be a root system in $\mathbb{R}^n$ and let $M= \mathbb{Z}\langle \Phi \rangle$.
 We have defined a toric arrangement $\mathcal{A}_{\Phi}= \{A_{\alpha}\}_{\alpha \in \Phi}$ in $X = \mathrm{Hom}(M,\mathbb{C}^*)$ associated to $\Phi$ by setting
 \begin{equation*}
  A_{\alpha} = \{\chi \in X | \chi(\alpha)=1\}
 \end{equation*}
 for all $\alpha \in \Phi$. Similarly we define a hyperplane arrangement $\mathcal{B}_{\Phi} = \{B_{\alpha}\}_{\alpha \in \Phi}$ in $Y = \mathrm{Hom}(M,\mathbb{C})$
 by setting
 \begin{equation*}
  B_{\alpha} = \{\phi \in Y | \phi(\alpha)=0\}.
 \end{equation*}
 The arrangements $\mathcal{A}_{\Phi}$ and $\mathcal{B}_{\Phi}$ and their posets are related
 but to see how we shall change the perspective slightly.
 
 Let $V = M \otimes \mathbb{C}$. Fleischmann and Janiszczak \cite{fleischmannjaniszczak}
 define the poset $\mathscr{P}(\Phi)$ as the poset of linear spans in $V$ of subsets of $\Phi$,
 ordered by inclusion, and show that the posets $\mathscr{L}(\mathcal{B}_{\Phi})$ and $\mathscr{P}(\Phi)$ are isomorphic. 
 We follow them and define $\mathscr{R}(\Phi)$ as the poset
 of submodules of $M$ spanned by elements of $\Phi$, ordered by inclusion.
 In very much the same way we have that the posets $\mathscr{L}(\mathcal{A}_{\Phi})$ and $\mathscr{R}(\Phi)$
 are isomorphic.
 
 There is a surjective order preserving map $\rho: \mathscr{R}(\Phi) \to \mathscr{P}(\Phi)$
 sending a submodule $N \subset M$ to $N \otimes \mathbb{C}$.
 If we represent elements in $\mathscr{R}(\Phi)$ and $\mathscr{P}(\Phi)$ by their echelon basis matrices, 
 $\rho$ sends a matrix $C$ to the matrix obtained from $C$ by dividing each row
 by the greatest common divisor of its entries, i.e. sending a matrix to its saturation.
 We thus see that $\rho$ is an isomorphism if and only if every module in $\mathscr{R}(\Phi)$
 is saturated.
 For $\Phi=A_n$ this is indeed the case. The pivotal observation for proving this
 is the following lemma, which can be proven via a simple induction argument on the number of rows using Gaussian elimination.

 \begin{lem}
 \label{matlemma}
  Let $C$ be a binary matrix of full rank such that the $1$'s in each row of $C$
  are consecutive. Then each pivot element in the row reduced echelon
  matrix (over $\mathbb{Z}$) obtained from $C$ is $1$.
 \end{lem}
 
 If one expresses the positive roots of $A_n$ in terms of the simple roots $\beta_i$, 
 each root is a vector of zeros and ones with all ones consecutive.
 It thus follows from Lemma~\ref{matlemma} that the modules in $\mathscr{R}(A_n)$
 are saturated and therefore we have that the map $\rho: \mathscr{R}(A_n) \to \mathscr{P}(A_n)$
 is an isomorphism of posets.
 
 \begin{rem}
  After the appearance of the first preprint of this paper, Bibby \cite{bibby} showed
  the stronger result that the intersection poset of a hyperplane arrangement,
  a toric arrangement and of an abelian arrangement associated to $A_n$ is isomorphic to
  the partition lattice. 
 \end{rem}
 
 The above observation is one key ingredient in our third proof of Theorem~\ref{anpoincpolprop}.
 Another key ingredient is the following theorem of MacMeikan.
 
 \begin{thm}[MacMeikan \cite{macmeikan}]
\label{macmeikanthm}
 Let $\mathcal{A}=\left\{A_i\right\}_{i \in I}$ be a toric arrangement or an arrangement of hyperplanes.
 Then
 \begin{equation*}
  P(\undcal{X}{A},t) = \sum_{Z \in \pos{L}{A}} \mu(Z) (-t)^{\cd(Z)} P(Z,t)(g),
 \end{equation*}
 where $\cd(Z)$ denotes the codimension of $Z$ in $X$ and $\mu$ denotes the M\"obius function of $\pos{L}{A}$.
\end{thm}

MacMeikan's result is in fact quite a bit stronger but this version is enough for our purposes.
Note that if $\mathcal{A}$ is a toric arrangement, then each $Z \in \pos{L}{A}$ is a disjoint
union of tori and $P(Z,t)=c_Z(1+t)^{\mathrm{dim}(Z)}$ where $c_Z$ is the number of components of
$Z$. Note also that if $\mathcal{A}$ is a hyperplane arrangement, then then each $Z \in \pos{L}{A}$ is an
affine space and therefore $P(Z,t)=1$ and Theorem~\ref{macmeikanthm} thus reduces to the Orlik-Solomon formula in this case.

 \begin{proof}[Proof 3]
 Let $\mu_{\mathscr{P}}$ denote the M\"obius function of $\mathscr{P}(\Phi)$.
 Theorem~\ref{macmeikanthm} and Equation~(\ref{arnoldformula}) tell us that
 \begin{equation*}
  \sum_{V \in \mathscr{P}(\Phi)} \mu_{\mathscr{P}}(V) \cdot (-t)^{\mathrm{dim}(V)} = \prod_{i=1}^n (1+i \cdot t).
 \end{equation*}
 By equating the coefficients of $t^r$ we get
 \begin{align*}
  \sum_{\stackrel{V \in \mathscr{P}(\Phi)}{\mathrm{dim}(V)=r}} \mu_{\mathscr{P}}(V)(-1)^{r} & = 
  \sum_{\stackrel{I \subseteq \{1, \ldots, n\}}{|I|=r}} \prod_{i \in I} i = \\
  & = e_r(1, \ldots, n),
 \end{align*}
 where $e_r$ denotes the $r$th elementary symmetric polynomial.
 We saw above that the map $\rho: \mathscr{R}(\Phi) \to \mathscr{P}(\Phi)$ is an isomorphism
 of posets. Hence, if $\mu_{\mathscr{R}}$ denotes the M\"obius function of $\mathscr{R}(\Phi)$
 we have that 
 \begin{equation*}
 \mu_{\mathscr{R}}(N)=\mu_{\mathscr{P}}(\rho(N)).
 \end{equation*}
 It thus follows that
 \begin{equation}
 \label{mobsum}
   \sum_{\stackrel{N \in \mathscr{R}(\Phi)}{\mathrm{rk}(N)=r}} \mu_{\mathscr{R}}(N)(-1)^{r} =
   e_r(1, \ldots, n).
 \end{equation}

 If we apply Theorem~\ref{macmeikanthm} to $X_{\Phi}$, we obtain
 \begin{align*}
  P(X_{\Phi},t) & = \sum_{N \in \mathscr{R}(\Phi)} \mu_{\mathscr{R}}(N)(-t)^{\mathrm{rk}(N)} \cdot (1+t)^{n-\mathrm{rk}(N)} \\
  & = \sum_{r=0}^n t^{r} \cdot (1+t)^{n-r} \sum_{\mathrm{rk}(N)=r}  \mu_{\mathscr{R}}(N) \cdot (-1)^{r}.
 \end{align*}
  If we use Equation~(\ref{mobsum}) we now see that the coefficient of $t^k$ in
 $P(X_{\Phi},t)$ is
\begin{equation*}
 \sum_{j=0}^k \binom{n-j}{k-j} \cdot e_j(1, \ldots, n).
\end{equation*}
 The coefficient of $t^k$ in $\prod_{i=1}^n (1+(i+1) \cdot t)$ is
 \begin{align*}
  \sum_{\stackrel{I \subseteq \{1, \ldots, n\}}{|I|=k}} (i_1+1) \cdots (i_k+1) 
  & =  \sum_{\stackrel{I \subseteq \{1, \ldots, n\}}{|I|=k}} \sum_{j=0}^k e_j(i_1, \ldots, i_k) = \\
  & = \sum_{j=0}^k \sum_{\stackrel{I \subseteq \{1, \ldots, n\}}{|I|=k}} e_j(i_1, \ldots, i_k) = \\
  & = \sum_{j=0}^k \binom{n-j}{k-j} \cdot e_j(1, \ldots, n).
 \end{align*}
 This proves the claim.
 \end{proof}


\subsection*{Acknowledgments}
Some of this paper is based on parts of my thesis \cite{bergvallthesis}.
I would like to thank Carel Faber and Jonas Bergstr\"om for helpful discussions and comments, 
Federico Ardila for pointing out a preprint 
of the interesting paper \cite{ardilacastillohenley}, Emanuele Delucchi for interesting discussions
and an anonymous referee for many helpful comments and suggestions.

\bibliographystyle{plain}

\renewcommand{\bibname}{References} 

\bibliography{references} 

\end{document}